\documentclass[preprint]{elsarticle}
\usepackage{amsmath,amssymb,amsthm}
\usepackage{graphicx}
\usepackage[labelsep=none]{caption}
\usepackage{float}

\newtheorem{theorem}{Theorem}[section]
\newtheorem{lemma}{Lemma}[section]
\newtheorem{remark}{Remark}[section]

\newtheorem{cor}{Corollary}[section]
\newtheorem{prop}{Proposition}[section]
\newtheorem{defn}{Definition}[section]

\begin{document}
\begin{frontmatter}

\title{Polygon Vertex Extremality and Decomposition of Polygons}
\author{Wiktor J. Mogilski}
\date{}

\begin{abstract}
In this paper, we show that if we decompose a polygon into two
smaller polygons, then by comparing the number of extremal vertices
in the original polygon versus the sum of the two smaller polygons,
we can gain at most two globally extremal vertices in the smaller
polygons, as well as at most two locally extremal vertices. We then will derive two discrete
Four-Vertex Theorems from our results.
\end{abstract}
\end{frontmatter}
\section{Introduction}

There are many notions of extremality in polygons, the earliest
appeared circa 1813 in \cite{cau}. Recently, a very natural type of
extremality was introduced in \cite{mus1}, one which very
consistently adhered to that of curvature in the smooth case. A
closely related global analogue had already appeared much earlier,
and it as well has a smooth and discrete interpretation. While it is
debatable to whom we attribute this discrete global notion of
extremality, closely related ideas are present in \cite{bos}.

In this paper, we will expound on these two types of extremality by
providing a few observations and facts to build intuition. We will then
discuss the notion of decomposing a polygon and investigate how
this impacts our two types of extremality. We then derive fresh
results relating the number of extremal vertices
of the larger polygon versus the two smaller polygons of
decomposition. While our results will be relevant geometrically on
their own, we will observe that they are closely tied to two
discrete Four-Vertex Theorems pertaining to our two types of
extremality, which follow almost immediately from our stronger results.

We note that we will skip proofs of the more simple results. All
results in this paper are considered with much more detail in
\cite{mog}.

\section{Global and Local Extremality}

We denote by $P$ a polygonal curve, which is a simple piecewise
linear curve with vertices $V_{1},V_{2},...,V_{n}$. When we speak of
a closed polygonal curve, we will refer to it as a polygon. Also, we
restrict our consideration simply to the planar case and all
indices will be taken modulo the number of vertices of the polygonal
curve. The following definition was coined in \cite{pak}:

\begin{defn}
We say that a polygonal curve is generic if the maximal number of
vertices that lie on a circle is three and no three vertices are
collinear.
\end{defn}

Observe that all regular polygons are not generic.

\begin{defn}
Let $C_{ijk}$ be a circle passing through any three vertices
$V_{i}$, $V_{j}$, $V_{k}$ of a polygonal curve. We say that
$C_{ijk}$ is empty if it contains no other vertices of the polygonal
curve in its interior, and we say that it is full if it contains all
of the other vertices of the polygonal curve in its interior.
\end{defn}

For simplicity, we will denote a circle passing through consecutive
vertices $V_{i-1},$ $V_i$ and $V_{i+1}$ by $C_i.$

\begin{defn}
We call a full or empty circle $C_i$ an extremal circle. We refer to
the corresponding vertex $V_i$ as a globally extremal vertex.
\end{defn}

Some of our results will use triangulation arguments. Consider all
of the empty circles passing through any three distinct points of a
polygon. In \cite{del} B. Delaunay shows that the triangles formed by
each of the three points corresponding to an empty circle form a
triangulation of the polygon $P$. This triangulation is called a
\textit{Delaunay triangulation}.

Analogously, if we assume convexity on our polygon and
consider the full circles passing through any given three
points, the triangles given by each of the three points
corresponding to a full circle also form a triangulation. This
triangulation is commonly known as the \textit{Anti-Delaunay
triangulation}.

\begin{defn}A vertex $V_{i}$ is said to be positive if the left angle with respect to orientation, $\angle V_{i-1}V_{i}V_{i+1}$, is at most $\pi$. Otherwise, it is said to be negative.
\end{defn}

\begin{defn}[Discrete Curvature]Assume that a vertex $V_i$ is positive. We say that the curvature of the vertex $V_{i}$ is greater than the curvature at $V_{i+1}$ ($V_{i}\succ V_{i+1})$ if the vertex $V_{i+1}$ is positive and $V_{i+2}$ lies outside the circle $C_{i}$ or if the vertex $V_{i+1}$ is negative and $V_{i+2}$ lies inside the circle $C_{i}$.

By switching the word ``inside" with the word ``outside" in the
above definition (and vice-versa), we obtain that $V_{i}\prec
V_{i+1}$, or that the curvature at $V_{i}$ is less than the
curvature at $V_{i+1}$. In the case that the vertex $V_{i}$ is
negative, simply switch the word ``greater" with the word ``less",
and the word ``outside" by the word ``inside".
\end{defn}

\begin{defn}A vertex $V_{i}$ of a polygonal line $P$ is locally extremal if

\centerline{$V_{i-1}\prec V_{i} \succ V_{i+1}$ or $V_{i-1}\succ
V_{i} \prec V_{i+1}$.}
\end{defn}

\begin{remark}
If we assume convexity on our polygon and observe the definition of locally extremal vertices closely, we simply are
considering the position of the vertices $V_{i-2}$ and $V_{i+2}$
with respect to the circle $C_{i}.$ Our vertex $V_{i}$ will be
locally extremal if and only if both vertices $V_{i-1}$ and $V_{i+2}$ lie inside or outside the
circle $C_{i}.$
\end{remark}

When defining global extremality, we discussed empty and full
extremal circles. If a circle $C_i$ is empty, then we say that the
corresponding vertex $V_i$ is maximal. If $C_i$ is full, then we
say $V_i$ is minimal. Analogously for locally extremal vertices, we call a
vertex maximal if $V_{i-1}\prec V_{i} \succ V_{i+1}$ and minimal if
$V_{i-1}\succ V_{i} \prec V_{i+1}$.

We denote the number of globally maximal-extremal vertices of a
polygonal curve $P$ by $s_{-}(P)$ and globally minimal-extremal
vertices by $s_{+}(P)$ to be consistent with \cite{bos}.
For locally extremal vertices, we will attribute the
notation $l_{-}(P)$ and $l_{+}(P)$, respectively.

\begin{prop}
\label{prop:maxmin} Let $P$ be a generic convex polygon. Then
$$l_{+}(P)=l_{-}(P).$$
\end{prop}

\begin{remark}
The proof of this fact immediately follows by carefully observing the definition of locally extremal vertices. Note
that it was very important for us to include the assumption that our
polygon is generic, since this eliminates the possibility of having
two extremal vertices adjacent to each other. Also, it is easy to
see that the equality $s_{+}(P)=s_{-}(P)$ does not hold. In fact, we
cannot form any relationship between globally maximal-extremal and
globally minimal-extremal vertices.
\end{remark}

\begin{prop}
\label{prop:globaltolocal} Let $P$ be a generic convex polygon. If
$V_i$ is a globally extremal vertex, then $V_{i}$ is a locally
extremal vertex.
\end{prop}

This result follows immediately from the observation made in Remark
2.1.

\begin{prop}
\label{prop:quad} Let $P$ be a generic convex quadrilateral. Then
$P$ has four globally extremal and locally extremal vertices.
\end{prop}
\begin{proof}
For globally extremal vertices, we apply a Delaunay triangulation to
$P$, which immediately yields two globally maximal-extremal vertices. We then apply an Anti-Delaunay triangulation to $P$, which yields
two minimal-extremal vertices. Proposition
\ref{prop:globaltolocal} then yields the result for locally extremal vertices.
\end{proof}

While the following proposition is technical yet quite obvious,
it will be a vital proposition that will be used frequently to prove
our main results.

\begin{prop}
\label{prop:circleprop} Let $A$, $B$, $C$ and $X$ be four points in
the plane in a generic arrangement, $C_B$ be the corresponding
circle passing through $A$, $B$ and $C$, and let $C_A$ be the circle
passing through the points $X$, $A$ and $B$. We denote by
$\widetilde{C}_A$ and $\widetilde{C}_B$ the open discs bounded by
$C_A$ and $C_B$, respectively. Denote by $H^{+}_{AB}$ the half-plane
formed by the infinite line $AB$ containing the point $C$ and by
$H^{-}_{AB}$ the half-plane formed by the infinite line $AB$ not
containing the point $C$. If $X$ lies in $\widetilde{C}_B\bigcap
H^{+}_{AB}$, then $C$ lies in $H^{+}_{AB}\setminus \widetilde{C}_A$.
If $X$ lies in $H^{+}_{AB}\setminus \widetilde{C}_B$, then $C$ lies
in $\widetilde{C}_A$. Analogously, if $X$ lies in
$\widetilde{C}_B\bigcap H^{-}_{AB}$, then $C$ lies in
$\widetilde{C}_A$. If $X$ lies in $H^{-}_{AB}\setminus
\widetilde{C}_B$, then $C$ lies in $H^{+}_{AB}\setminus
\widetilde{C}_A$.
\end{prop}
\begin{proof}
The proof is a simple verification of the the situation restricted around the origin
and solving corresponding systems of equations.
\end{proof}

\section{Globally Extremal Vertices and Decomposition of Polygons}

\begin{defn}
We say an edge or diagonal of a polygon is Delaunay if there exists an
empty circle passing through the corresponding vertices of that
edge or diagonal. If there exists a full circle passing through the vertices of
this edge or diagonal, then we say the edge or diagonal is Anti-Delaunay.
\end{defn}

\begin{remark}
Note that a triangulation of a polygon where every edge and diagonal
is Delaunay is a Delaunay Triangulation. Similarly, if every edge
and diagonal of a triangulation is Anti-Delaunay, then we have an
Anti-Delaunay triangulation.
\end{remark}

So what exactly does it mean to decompose a polygon? Here the
notion of decomposing a polygon will simply be the cutting of a
polygon $P$ by passing a line segment through any two vertices so
that the line segment lies in the interior of the polygon. We will
call this line segment a \textit{diagonal}. Also, we denote the
two new polygons formed by a decomposition by $P_1$ and $P_2$ and
require that they each have at least four vertices. By this
last requirement it automatically follows that $P$ must have at least
six vertices to successfully perform a decomposition.

\begin{theorem}
\label{theorem:strongglobal} Let $P$ be a generic convex polygon
with six or more vertices and $P_1$ and $P_2$ be the resulting polygons of a decomposition of $P$. Assume that the diagonal of this
decomposition is Delaunay. Then
$$s_{-}(P)\geq s_{-}(P_1)+s_{-}(P_2)-2.$$
Analogously, if the diagonal is Anti-Delaunay, then
$$s_{+}(P)\geq s_{+}(P_1)+s_{+}(P_2)-2.$$
\end{theorem}
\begin{proof}
We begin by applying a Delaunay triangulation to $P$, $P_1$ and $P_2$. Noticing that the diagonal is Delaunay for $P_1$ and $P_2$, as well as $P$ by assumption, we obtain our first inequality. For the second inequality we mimic this argument, instead applying an Anti-Delaunay triangulation.
\end{proof}

It turns out that from the above result, we can derive a very nice
geometric corollary. First, we need two small lemmas.

\begin{lemma}
\label{lemma:triangulation} Let $P$ be a convex polygon with seven
or more vertices and let $T(P)$ be a triangulation of $P$. Then,
there exists a diagonal of our triangulation such that, if we apply
a decomposition of $P$ using this diagonal, then both $P_1$ and
$P_2$ have four or more vertices.
\end{lemma}

This result is clear, and follows immediately by an induction
argument on the number of vertices.

\begin{remark}
It is obvious that this result does not hold if $n=6$. In fact, it is easy to find a convex polygon whose Delaunay Triangulation does not satisfy Lemma \ref{lemma:triangulation}, hence the need for one more lemma.
\end{remark}

\begin{lemma}
\label{lemma:decompglobal} Let $P$ be a generic convex polygon with
six vertices and let $P_1$ and $P_2$ be the resulting polygons of a
decomposition. Then
$$s_{-}(P)\geq s_{-}(P_1)+s_{-}(P_2)-2$$ and $$s_{+}(P)\geq s_{+}(P_1)+s_{+}(P_2)-2.$$
\end{lemma}

\begin{proof}
Since we have no guarantee that our diagonal is Delaunay, we cannot mimic the proof of Theorem \ref{theorem:strongglobal}. We observe that, since $P$ is generic, $P_1$ and $P_2$ are generic as well. Moreover, $P_1$ and $P_2$ are quadrilaterals. By applying Proposition \ref{prop:quad} to $P_1$ and $P_2$, we prove our assertion.
\end{proof}

\begin{cor}[The Global Four-Vertex Theorem]
\label{cor:newglobal} Let $P$ be a generic convex polygon with six
or more vertices. Then $$s_{+}(P)+s_{-}(P)\geq 4.$$
\end{cor}
\begin{proof}
We will prove the result by induction on the number of
vertices of $P$. We first consider the base case $n=6$, noticing if we apply a decomposition to
$P$, then $P_1$ and $P_2$ are both quadrilaterals. By
Proposition \ref{prop:quad}, we obtain that $P_1$ and $P_2$ each have four
globally extremal vertices. It follows from Lemma
\ref{lemma:decompglobal} that $P$ has four globally extremal
vertices.

We now consider the case where $n\geq7$. We begin by applying a
Delaunay triangulation to $P$. By Lemma \ref{lemma:triangulation},
it follows that there exists a diagonal $d$ such that when we
decompose $P$ by this diagonal, $P_1$ and $P_2$ each have four or
more vertices. Since our diagonal corresponds to a Delaunay triangulation, it follows
that $d$ is Delaunay. Since $P_1$ and $P_2$
have less vertices than $P$, we apply the inductive assumption to obtain
$s_{-}(P_1)\geq 2$ and $s_{-}(P_2)\geq 2$. Applying this to Theorem
\ref{theorem:strongglobal}, we obtain $s_{-}(P)\geq 2$. An analogous
argument using an Anti-Delaunay triangulation and Theorem
\ref{theorem:strongglobal} yields $s_{+}(P)\geq 2$. So $s_{+}(P)+s_{-}(P)\geq 4$, proving the assertion.
\end{proof}

\section{Locally Extremal Vertices and Decomposition of Polygons}

When considering locally extremal vertices, it is easy to see that
the only vertices affected by a decomposition of a
polygon will be the vertices on the diagonal of decomposition and the
neighboring vertices.

\begin{figure}[H]
\centerline{\includegraphics[scale=0.175]{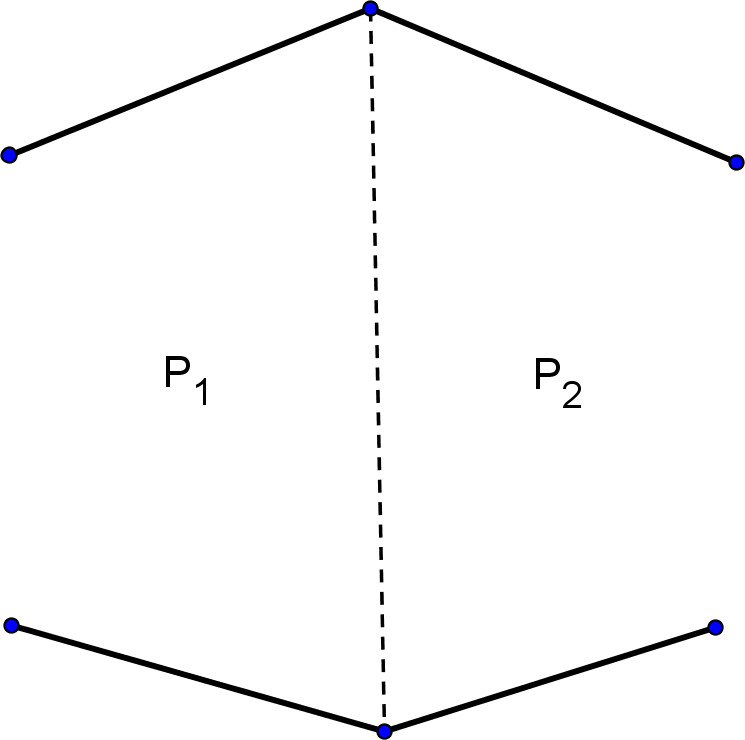}}
\caption[Proving Theorem \ref{thm:localineq}]{}
\end{figure}

This means that we have a total of six vertices impacted by a decomposition, leading us to a feasible case-by-case analysis. Before proving our main result, we need a few lemmas.

\begin{lemma}
\label{lemma:localdec1} Let $P$ be a generic convex polygon and $P_1$ and $P_2$ the resulting polygons of a decomposition. Denote the vertices of the diagonal by $B$ and $D$, the neighboring vertex of $B$ in $P_1$ by A, and the neighboring vertex of $B$ in $P_2$ by $C$. Assume that $A$ is locally
maximal-extremal in $P_1$ but not in $P$, and that $C$ is locally
maximal-extremal for $P_2$ but not in $P$. Then, $B$ is a locally
maximal-extremal vertex for $P$.
\end{lemma}
\begin{proof}
Let $X$ be the neighbor of $A$ in $P_1$ and $Y$ be the neighbor of
$C$ in $P_2$. Denote the circle passing through vertices $A$, $B$
and $C$ by $C_B$, the circle passing through vertices $X$, $A$ and
$B$ by $C_A$, and the circle passing through vertices $B$, $C$ and
$Y$ by $C_C$. Since $A$ is not maximal-extremal in $P$, it follows
that $A$ lies inside the circle $C_C$. By Proposition
\ref{prop:circleprop}, it follows that $Y$ lies outside of the
circle $C_B$. Since $C$ is not maximal-extremal in $P$, it follows
that $C$ lies inside the circle $C_A$. By Proposition
\ref{prop:circleprop}, it follows that $X$ lies outside of the
circle $C_B$. The following figure illustrates the situation:

\begin{figure}[H]
\centerline{\includegraphics[scale=0.55]{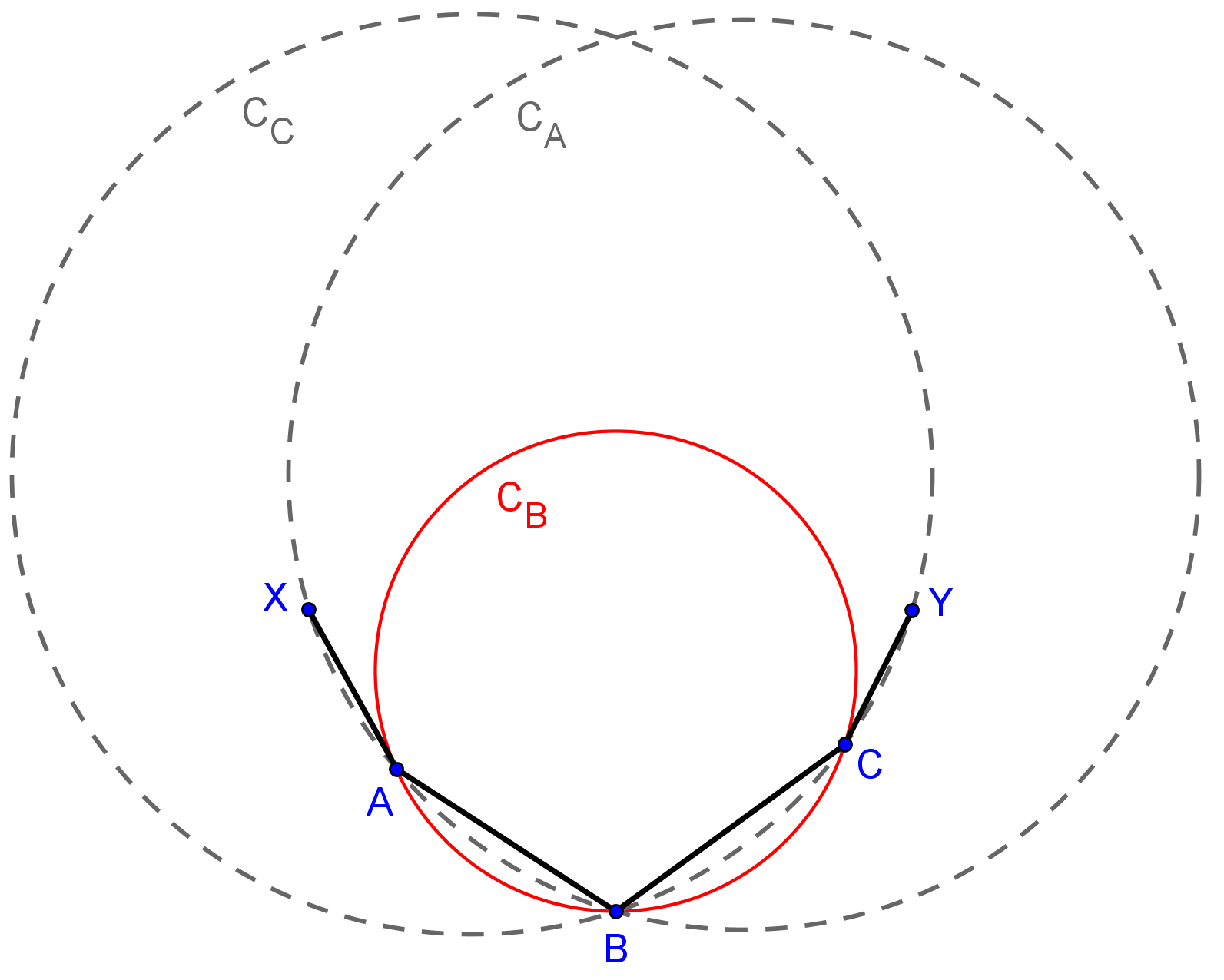}}
\caption[Proving Lemma \ref{lemma:localdec1}]{}
\end{figure}

Since both $X$ and $Y$ lie outside of the circle
$C_B$, $B$ is maximal-extremal in $P$.
\end{proof}

\begin{lemma}
\label{lemma:localdec2} Let $P$ be a generic convex polygon and $P_1$ and $P_2$ the resulting polygons of a decomposition. Denote the vertices of the diagonal by $B$ and $D$, the neighboring vertex of $B$ in $P_1$ by A, and the neighboring vertex of $B$ in $P_2$ by $C$. Assume that $A$ is locally
maximal-extremal in $P_1$ but not in $P$, and that $B$ is locally
maximal-extremal in $P_2$. Then, $B$ is locally maximal-extremal in
$P$.
\end{lemma}
\begin{proof}
For simplicity, consider the following figure, which will illustrate our configuration of points and circles:

\begin{figure}[H]
\centerline{\includegraphics[scale=0.65]{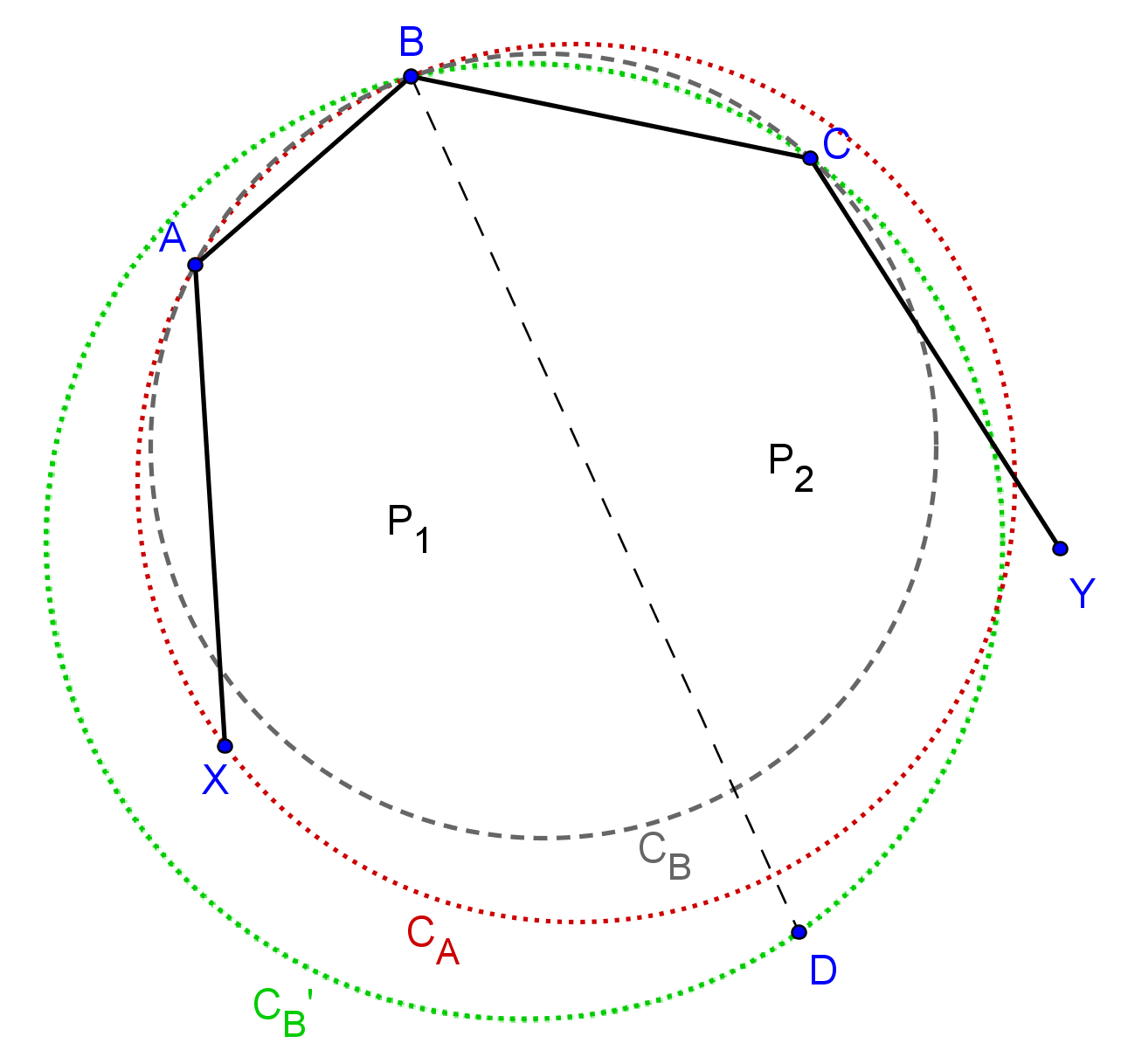}}
\caption[Proving Lemma \ref{lemma:localdec2}]{}
\end{figure}

Let $X$ be the neighbor of $A$ in $P_1$ and $Y$ be the neighbor of $C$ in $P_2$. Denote
by $C_A$ the circle passing through vertices $X$, $A$, and $B$.
Since $A$ is maximal-extremal in $P_1$, it follows that $D$ lies
outside of the circle $C_A$. Since $A$ is not maximal-extremal in
$P$, it follows that $C$ must lie inside the circle $C_A$. Now,
denote the circle passing through vertices $A$, $B$, and $C$ by
$C_B$. Our goal is to show that vertices $X$ and $Y$ lie outside of
the circle $C_B$.

A quick application of Proposition \ref{prop:circleprop} to points
$X$, $C$, $A$ and $B$ yields that $X$ lies outside of $C_B$, so we
need to show that $Y$ lies outside of the circle $C_B$. Denote by
$C_B'$ the circle passing through the points $C$, $B$ and $D$. We
will show that if $Y$ lies outside of $C_B'$, then it lies outside
of $C_B$. To do this, we first must show that $A$ lies inside the
circle $C_B'$.

Consider the circles $C_A$ and $C_B'$. These circles intersect at
two points, point $B$ and some other point, say $Z$. Since $D$ lies
outside of the circle $C_A$, it follows by an application of
Proposition \ref{prop:circleprop} to points $A$, $D$, $B$ and $Z$
that $A$ lies inside the circle $C_B'$.

Lastly, consider the circles $C_B$ and $C_B'$. These two circles
intersect at the points $B$ and $C$. Since $A$ lies inside the
circle $C_B'$, it follows from applying Proposition
\ref{prop:circleprop} to points $A$, $D$, $B$ and $C$ that $D$ lies
outside of the circle $C_B$.

Now, since $B$ is maximal-extremal in $P_2$, it follows that $Y$
lies outside of $C_B'$. By our above observation, it follows
immediately from Proposition \ref{prop:circleprop} that $Y$ lies
outside of $C_B$. Since points $X$ and $Y$ both lie outside of the
circle $C_B$, it follows that $B$ is maximal-extremal in $P$.
\end{proof}

\begin{lemma}
\label{lemma:localdec3} Let $P$ be a generic convex polygon and $P_1$ and $P_2$ the resulting polygons of a decomposition. Denote the vertices of the diagonal by $B$ and $D$, the neighboring vertex of $B$ in $P_1$ by A, and the neighboring vertex of $B$ in $P_2$ by $C$. Assume that $A$ is locally
maximal-extremal for $P_1$ and $D$ is locally maximal-extremal for
both $P_1$ and $P_2$, but not for $P$. Then $A$ is locally
maximal-extremal for $P$.
\end{lemma}
\begin{proof}
Let $X$ be the neighbor of $A$ in $P_1$, $E$ be the neighbor of $D$
in $P_1$, and $F$ be the neighbor of $D$ in $P_2$. Denote by
$C_{D1}$ the circle passing through vertices $B$, $D$ and $E$, by
$C_{D2}$ the circle passing through vertices $B$, $E$ and $F$, and
by $C_A$ the circle passing through vertices $X$, $A$ and $B$. The following figure illustrates our configuration:

\begin{figure}[H]
\centerline{\includegraphics[scale=0.95]{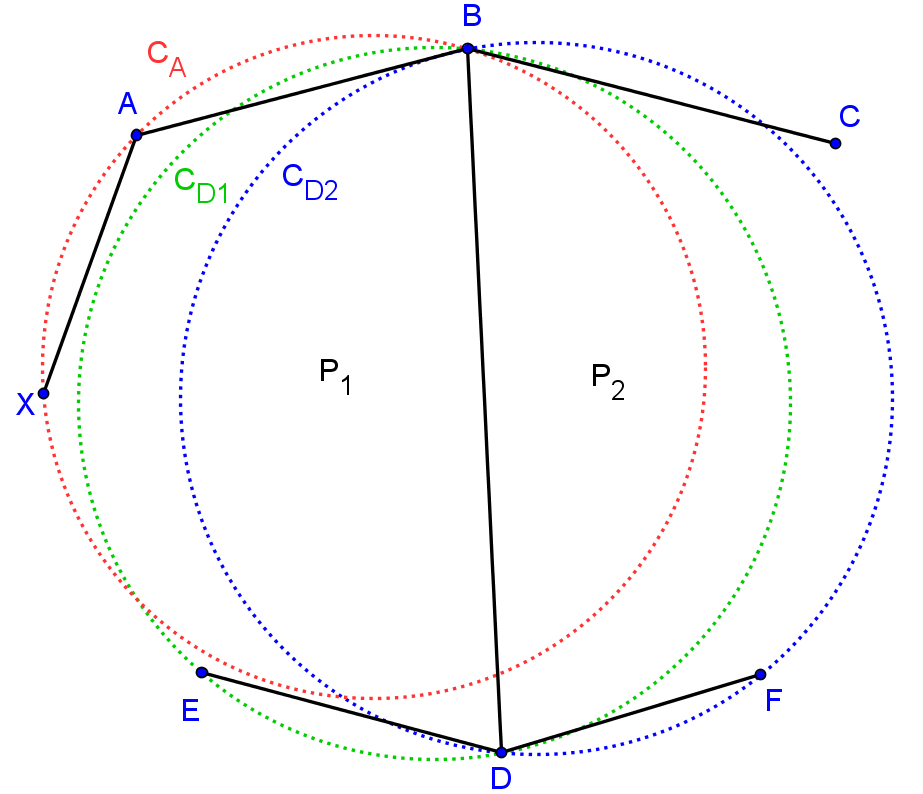}}
\caption[Proving Lemma \ref{lemma:localdec3} (1)]{}
\end{figure}

Our goal is to show that vertex $C$ lies outside of the circle
$C_A$. We will do this by showing that if $C$ lies outside the
circle $C_{D2}$, then it also lies outside of circle $C_A$. Since
$A$ is maximal-extremal in $P_1$, it follows that $D$ lies outside
of $C_A$. Since $D$ is maximal-extremal in $P_1$, it follows that
$A$ lies outside of circle $C_{D1}$. By a similar argument used in the
previous lemma, it follows that if $C$ lies outside of $C_{D1}$ then
it lies outside of $C_A$. The following figure illustrates this situation:

\begin{figure}[H]
\centerline{\includegraphics[scale=1.0]{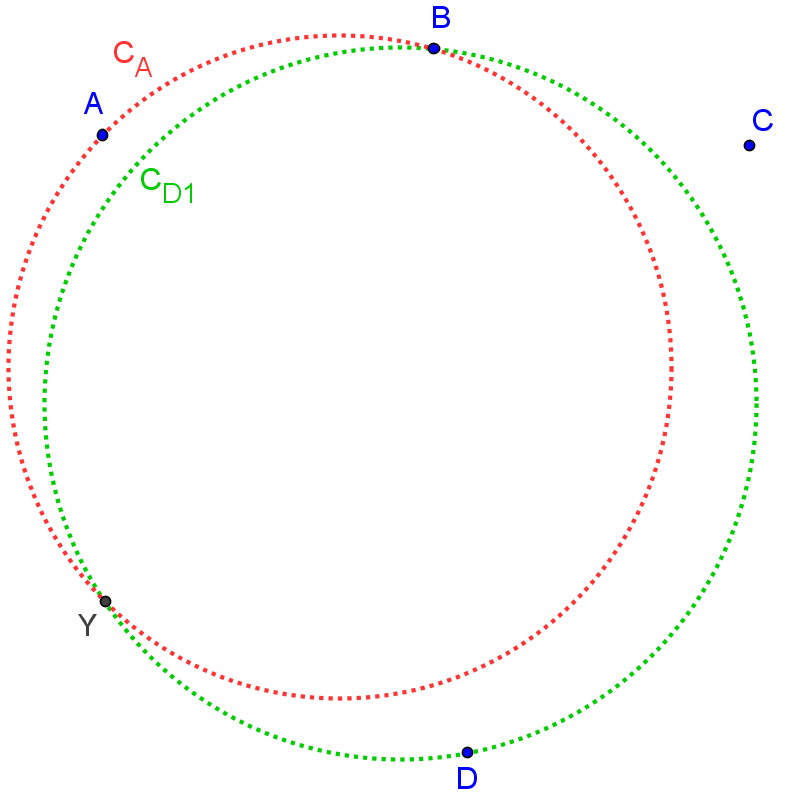}}
\caption[Proving Lemma \ref{lemma:localdec3} (2)]{}
\end{figure}

It remains to show that $C$ lies outside
of $C_{D1}$. Consider the circles $C_{D1}$ and $C_{D2}$. Since $D$ is
maximal-extremal in $P_2$, it follows that $C$ lies outside of the
circle $C_{D2}$. If we show that $C$ also lies outside of $C_{D1}$,
then we are done. To do this, we will heavily use the fact that $D$
is not maximal-extremal in $P$. We will show that if $E$ lies inside
the circle $C_{D2}$ or if $F$ lies in $C_{D1}$, then $D$ is
maximal-extremal in $P$, contradicting our assumption.

It is enough just to check this for $E$. Denote the circle passing
through vertices $E$, $D$ and $F$ by $C_D$. If $E$ lies inside the
circle $C_{D2}$, then applying Proposition \ref{prop:circleprop} to
points $E$, $D$, $F$ and $B$ yields that $B$ lies outside of the
circle $C_D$. Similarly, it follows by Proposition
\ref{prop:circleprop} that $F$ lies inside the circle $C_{D1}$.

Now denote by $E'$ the neighbor of $E$ and by $F'$ the neighbor of
$F$. The following figure illustrates the situation:

\begin{figure}[H]
\centerline{\includegraphics[scale=1.1]{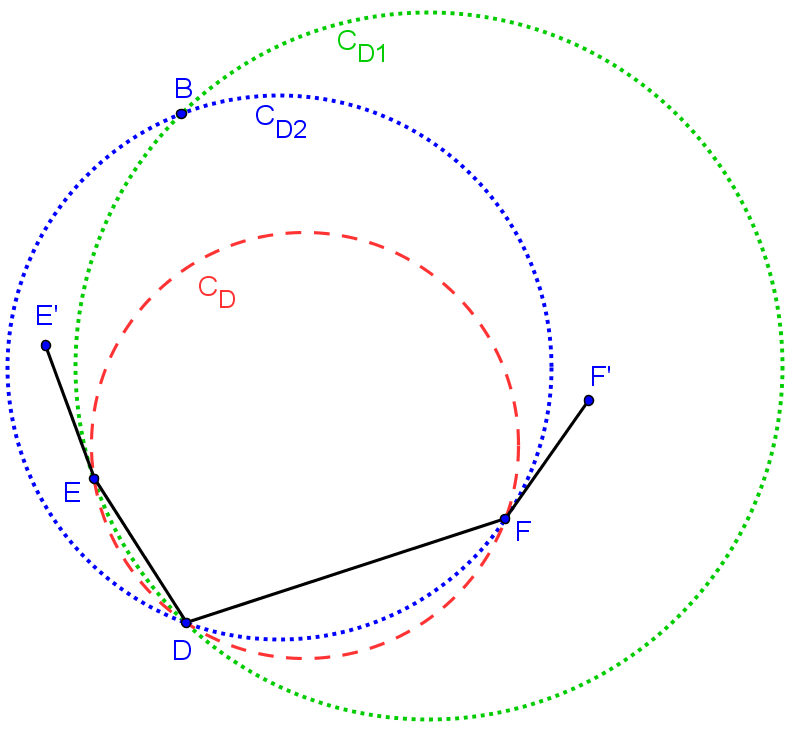}}
\caption[Proving Lemma \ref{lemma:localdec3} (3)]{}
\end{figure}

Since $D$ is maximal-extremal in $P_1$, it follows that $E'$
lies outside of the circle $C_{D1}$. Similarly, since $D$ is
maximal-extremal in $P_2$, it follows that $F'$ lies outside the
circle $C_{D2}$. Now, recall that $B$ lies outside of the circle
$C_D$. Proposition \ref{prop:circleprop} applied to points $E'$,
$B$, $D$ and $D$ tells us that $E'$ lies outside of $C_D$. A similar
argument yields that $F'$ also lies outside of $C_D$. So, we obtain
that $D$ is maximal-extremal in $P$, a contradiction.

So now we know that $E$ must lie outside of the circle $C_{D2}$.
Proposition \ref{prop:circleprop} applied to points $F$, $E$, $B$
and $D$ now tells us that $F$ lies outside of the circle $C_{D1}$.
So, if $C$ were to lie outside of circle $C_{D2}$, then it would
also lie outside of the circle $C_{D1}$. But earlier we showed that
if $C$ would lie outside of circle $C_{D1}$, then $C$ would lie
outside of the circle $C_A$. Indeed, by assumption, $C$ lies outside
of $C_{D2}$ and hence outside of $C_A$. Since $A$ is
maximal-extremal in $P_1$, it also follows that $X$ lies outside of
the circle $C_A$. Therefore $A$ is maximal-extremal in $P$.
\end{proof}

\begin{theorem}
\label{thm:localineq} Let $P$ be a generic convex polygon with at
least 6 vertices and let $P_1$ and $P_2$ be the resulting polygons
of a decomposition. Then
$$l_{-}(P)\geq l_{-}(P_1)+l_{-}(P_2)-2.$$
\end{theorem}
\begin{proof}
We note that only six vertices are affected by a decomposition from
the local point of view: the vertices of the diagonal and
the neighbors of those vertices. So, we eliminate the
cases which violate our inequality. It is easy to check that by the
symmetry of our cases, we only need to check three:

\noindent\textit{Case 1:} We gain two maximal-extremal vertices in
$P_1$, as well as $P_2$, but none of the six vertices are
maximal-extremal in $P$.

\noindent\textit{Case 2:} We gain two maximal-extremal vertices in
$P_1$ and gain two maximal-extremal vertex in $P_2$, and one of the
six vertices is maximal-extremal in $P$.

\noindent\textit{Case 3:} We gain two maximal-extremal vertices in
$P_1$ and gain one maximal-extremal vertex in $P_2$, and none of the
six vertices is maximal-extremal in $P$.

By checking the possible configurations of vertices in each of the
cases, we see that each case admits a configuration which is deemed
not feasible by one of the three preceding lemmas.
\end{proof}

\begin{cor}[The Local Four-Vertex Theorem]
\label{cor:newlocal} Let $P$ be a generic convex polygon with at
least six vertices. Then
$$l_{+}(P)+l_{-}(P)\geq 4.$$
\end{cor}
\begin{proof}
We apply induction on the number of vertices of $P$. For the case
where $n=6$, we know that if we apply a decomposition to $P$, then
both $P_1$ and $P_2$ will be quadrilaterals. Proposition \ref{prop:quad}
yields that $l_{-}(P_1)=l_{-}(P_2)=2$. Applying this to Theorem
\ref{thm:localineq} completes the proof for this case.

Now, assume that $n\geq 7$. We now apply induction
to the smaller polygons $P_1$ and $P_2$ to obtain that
$l_{-}(P_1)\geq 2$ and $l_{-}(P_2)\geq 2$. We now apply this to Theorem \ref{thm:localineq} to obtain that $l_{-}(P)\geq 2$. By Proposition \ref{prop:maxmin}, we obtain that $l_{+}(P)\geq 2$. Therefore $l_{+}(P)+l_{-}(P)\geq 4$, proving the assertion.
\end{proof}

\section{Acknowledgements}

The author would like to thank his advisor Oleg R. Musin for his
guidance and insight pertaining to the problem, as well as colleague
Arseniy Akopyan for thought provoking discussions.

\end{document}